\documentclass[12pt]{amsart}
\usepackage{amsmath,amsthm,amssymb}
\usepackage{times}
\usepackage{enumerate}

\usepackage{geometry}
\usepackage[english]{babel}
\usepackage{color}
\usepackage[normalem]{ulem}

\usepackage{enumitem}

\usepackage{pgfplots}
\usepackage{tikz}
\usetikzlibrary{calc, decorations.markings}
\usepackage{tikz-cd}[2011/10/21]

\usepackage{graphicx}
\usepackage{caption}

\definecolor{greenbean}{RGB}{199,237,204}

\definecolor{blue}{rgb}{0,0,1}

\definecolor{red}{rgb}{1,0,0}

\definecolor{green}{rgb}{0,1,0}

\definecolor{gray}{rgb}{.5,.5,.5}

\definecolor{yellow}{rgb}{1,1,.4}

\definecolor{purple}{rgb}{1,0,1}

\definecolor{gold}{rgb}{.5,.5,.2}

\usepackage{hyperref}
\hypersetup{
    bookmarks=true,         
    colorlinks=true,       
    linkcolor=cyan,          
    citecolor=cyan,        
    urlcolor=cyan           
}

\theoremstyle{plain}
\newtheorem{theorem}{Theorem}[section]
\newtheorem{lemma}[theorem]{Lemma}
\newtheorem{prop}[theorem]{Proposition}
\newtheorem{cor}[theorem]{Corollary}
\newtheorem{conj}[theorem]{Conjecture}

\theoremstyle{definition}
\newtheorem{Def}[theorem]{Definition}

\theoremstyle{remark}
\newtheorem{rmk}[theorem]{Remark}
\newtheorem{ex}[theorem]{Example}

\newcommand{\mult}{\textup{mult}}
\newcommand{\ord}{\textup{ord}}
\newcommand{\vol}{\textup{Vol}}
\newcommand{\ceil}[1]{\lceil #1 \rceil}

\newcommand{\mc}[1]{\mathcal{#1}}
\newcommand{\mb}[1]{\mathbb{#1}}
\renewcommand{\d}{\textup{d}}
\renewcommand{\O}{\mc{O}}

\newcommand{\im}{\textup{Im}}
\newcommand{\lct}{\textup{lct}}
\newcommand{\df}{\textup{def}}

\makeatletter
\newcommand{\hcell}[1]{\ifmeasuring@#1\else\omit$\displaystyle#1$\ignorespaces\fi}

\newcommand{\pushright}[1]{\ifmeasuring@#1\else\omit\hfill$\displaystyle#1$\fi\ignorespaces}

\newcommand{\pushleft}[1]{\ifmeasuring@#1\else\omit$\displaystyle#1$\hfill\fi\ignorespaces}

\makeatother

\title{Global Generation of Adjoint Line Bundles on Projective $5$-folds}

\author{Fei Ye} 

\address[Fei Ye]{Department of Mathematics and Computer Science,  QCC-CUNY, 222-05 56th Ave., Bayside, NY, 11364, USA}
\email{feye@qcc.cuny.edu}

\author{Zhixian Zhu}
\address[Zhixian Zhu]{ KIAS,
85 Hoegiro, Dongdaemun-gu,
Seoul 130-722,
Republic of Korea}
\email{zhixian@kias.re.kr}

\subjclass[2000]{Primary 14C20; Secondary 14F18, 14B05}

\date{}

\begin{document}

\maketitle

\begin{abstract}
Let $X$ be a smooth projective variety of dimension $5$ and $L$ be an ample line bundle on $X$ such that $L^5>7^5$ and $L^d\cdot Z\geq 7^d$ for any subvariety $Z$ of dimension $1\leq d\leq 4$. We show that $\O_X(K_X+L)$ is globally generated.

\end{abstract}

\section{Introduction}
Throughout this paper, we work over an algebraic closed field $k$ of characteristic zero. 

Geometric properties of pluricanonical and adjoint line bundles on surfaces and higher dimensional varieties have been extensively studied. Motivated by work of Kodaira \cite{Kodaira1968} and Bombieri \cite{Bombieri1973}, who studied pluricanonical maps of surfaces of general type, one wants to understand explicitly when  pluricanonical line bundles or more generally adjoint line bundles on higher dimension varieties are globally generated or very ample.  In \cite{Fujita1988}, Fujita posted the following conjecture.

\begin{conj}[Fujita]
Let $X$ be a smooth projective variety of dimension $n$ and $L$ be an ample line bundle on $X$. Then $\mathcal{O}_X(K_X+mL)$ is globally generated if $m\geq n+1$.  $\mathcal{O}_X(K_X+mL)$ is very ample if $m\geq n+2$.
\end{conj}

For curves, the conjecture follows from the Riemann-Roch theorem. For surfaces, the conjecture was proved by Reider \cite{Reider1988} using Bogomolov's instability theorem \cite{Bogomolov1978} on rank two vector bundles.  Unfortunately, so far this method is limited to surfaces.

For $n\geq 3$, towards Fujita's freeness conjecture, people follow a cohomological approach developed by Kawamata \cite{Kawamata1984}, Reid \cite{Reid1983} and Shokurov \cite{Shokurov1985}.  The basic idea is roughly the following. Given a point $x\in X$, under certain assumption on positivity of $D=mL$, we can find a general effective $\mathbb{Q}$-divisor $G$ linearly equivalent to $\lambda D$ with $0\leq \lambda<1$ such that the support $Z(G)$ of the multiplier ideal sheaf of $G$ is a normal subvariety containing $x$. We then apply Kawamata-Viehweg vanishing theorem to $(1-\lambda)D$ and reduce to a problem on $Z(G)$. 
Some effective results have been obtained in higher dimensions through this approach.

Demailly \cite{Demailly1993} proved that $\O_X(2K_X+12n^nL)$ is very ample using analytic tools. Koll\'ar \cite{Kollar1993} then proved that $\O_X(2(n+1)(n+2)!(K_X+(n+2)L))$ is globally generated.

When $n=3$, Ein and Lazarsfeld \cite{Ein1993a} proved, among others, that Fujita's freeness conjecture is true. Along with a similar approach, Fujita  \cite{Fujita1993} introduced an new technique, namely, calculation of the restrict volume of the pullback of the ample line bundle $L$ along the exception divisor $E$ of the blowing up at $x\in X$,  and improved Ein and Lazarsfeld's results.

By applying a theorem on extension of $L^2$ holomorphic functions, Angehrn and Siu \cite{Angehrn1995}, lowered the bound to a quadratic bound.

\begin{theorem}[Angehrn-Siu]Let $X$ be a smooth projective variety of dimension $n$ and $L$ be an ample line bundle.
If $m\geq \frac{1}{2}n(n+1)+1$, then  $\mathcal{O}_X(K_X+mL)$ is globally generated.
\end{theorem}
The method was later translated into algebraic setting by Koll\'ar \cite{Kollar1997}.

Following the work of  Fujita, and Angehrn and Siu, Helmke \cite{Helmke1999} obtained effective bounds in arbitrary dimensions in the spirit of Reider.
\begin{theorem}[Helmke]\label{thm-Helmke} Let $X$ be a smooth projective variety of dimension $n$,  $x\in X$ be a point and $L$ be an ample line bundle satisfying  the following conditions
\begin{enumerate}
\item       $L^n > n^n$;
\item      $L^{n-1}\cdot H\geq n^{n-1}$ for any hypersurface $H\subset X$ containing $x$;
\item   $L^d\cdot Z\geq \textup{mult}_xZ\cdot n^d$ for any any subvariety $Z\subset X$ of dimension $d=\dim Z <n$ which contains $x$.
\end{enumerate}
 Then  $\mathcal{O}_X(K_X+mL)$ is globally generated at $x$.
\end{theorem}

Helmke also showed that $\mult_xZ\leq \binom{n-1}{d-1}$ which makes the bound better than Angehrn and Siu's when dimension is not too large.  Balancing both results, Heier \cite{Heier2002} obtain the following.
\begin{theorem}[Heier]Let $X$ be a smooth projective variety of dimension $n$ and $L$ be an ample line bundle.
If $m\geq (e+\frac{1}{2})n^{\frac{4}{3}}+n^{\frac{2}{3}}+1$, then  $\mathcal{O}_X(K_X+mL)$ is globally generated.
\end{theorem}

Helmke's result together with those of Reider, Ein and Lazarsfeld, and Fujita leads to the following stronger conjecture.

\begin{conj}\label{H-K-K-Conj} Let $X$ be a smooth projective variety of dimension $n$,  $x\in X$ be a point and $L$ be an ample line bundle. Assume that
               $L^n > n^n$ and $L^d\cdot Z\geq n^d$ for any subvariety $Z\subset X$ of dimension $d=\dim Z <n$ which contains $x$. Then  $\mathcal{O}_X(K_X+mL)$ is globally generated at $x$.
\end{conj}

Conjecture \ref{H-K-K-Conj} is true for surfaces and 3-folds. When $n=4$, the following theorem of Kawamata \cite{Kawamata1997} confirmed Fujita's base point freeness conjecture but not  Conjecture \ref{H-K-K-Conj}.
\begin{theorem}[\cite{Kawamata1997}]\label{thm:Kawamata}
Let $X$ be a smooth projective variety of dimension 4,  $L$ be an ample line bundle on $X$ and $x\in X$ be a point. Assume that $L^d\cdot Z\geq 5^d$ for any subvariety $Z\subseteq X$ of dimension $d$ which contains $x$.
Then $\O_X(K_X+L)$ is globally generated at $x$.
\end{theorem}

In this paper, we show a Kawamata-type result on projective 5-folds by carefully analyzing upper bounds of deficit functions (Definition \ref{def:deficit}) and applying Helmke's induction criterion (Proposition \ref{Prop:Helmke'sInductionCriterion}).  More precisely, we prove that $\O_X(K_X+L)$ is globally generated if $\sqrt[5]{L^5}>7$ and $\sqrt[d]{L^d\cdot Z}\geq 7$ for any subvariety $Z$ of dimension $d\leq 4$ (Theorem \ref{Thm-dim-5}). The same proof also works on $4$-folds and reproduces Kawamata's result. 

The paper is organized as follow. We start by recalling some definitions and results from \cite{Helmke1999} and \cite{Ein1997} in Section \ref{Sec-deficit}. In section \ref{Sec-volume}, we study restricted volumes along exceptional divisors and their applications. We prove our result in Section \ref{Sec-proof}.

\noindent \textbf{Acknowledgements.} 
The authors would like to thank Professor Lawrence Ein for introducing this topic and his encouragement. They would also like to thank Professors Lawrence Ein, Mircea Musta\c{t}\u{a} and Karl Schwede for helpful discussions.  
The authors would like to thank the referee for carefully reading our manuscript and for his/her valuable comments which helped to improve the manuscript.
 
\section{The Deficit Function and Critical Varieties\label{Sec-deficit}}
In this section,  $X$ will be a smooth projective variety of dimension $n$ and $G$ will be an effective $\mathbb{Q}$-divisor on $X$. The multiplier ideal of $G$ is defined as $\mc{I}(G)=f_*\mathcal{O}_Y(K_{Y/X}-[f^*G])$, where $f: Y\to X$ is a log resolution of $G$. We denote by $Z(G)$ the scheme defined by $\mc{I}(G)$.

To prove the global generation of an adjoint line bundle, the key observation is the following lemma.

\begin{lemma}\label{cor:0-dimb-p-f}
Assume that $G$ is an effective $\mb{Q}$-divisor such that $Z(G)$ is $0$-dimensional at $x$. Let $A$ be an integral divisor such that $A-(K_X+G)$ is nef and big. Then the line bundle $\O_X(A)$ is globally generated at $x$.
\end{lemma}

To apply this lemma to the adjoint line bundle $K_X+L$, we will assume that $L^n>n^n$. By Riemann-Roch theorem, we can construct an effective divisor $G$ linearly equivalent to $\lambda L$ with $0<\lambda<1$ such that $x\in Z(G)$. Then we will show that $\dim Z(G)=0$ for a suitable choice of $G$. However, for a general choice of $G$, $\dim Z(G)$ may be positive, we have to find a way to modify the initial divisor $G$ to reduce the dimension.  For that purpose, we want $Z(G)$ to be ``minimal'' in the following sense. 

\begin{Def}\label{critical}
Let $x$ be a closed point in $X$. We say that $G$ is {\it critical} at $x$ if 
\begin{enumerate}
\item $x\in Z(G)$ and $x\not\in Z((1-\varepsilon)G)$ for any $0<\varepsilon <1$ and
\item $K_{Y/X}-[f^*G]=P-F-N$ for a log resolution $f: Y\to X$ such that $P$, $F$ and $N$ are effective with no common component,  $F$ is a prime divisor, $x\in f(F)$ and $N\cap f^{-1}(x)=\emptyset$.
\end{enumerate}
The component $F$ is called the critical component of $G$ at $x$ and $Z=f(F)$ is called the {\it critical variety} of $G$ at $x$.
\end{Def}

\begin{rmk}[Existence,  minimality and  normality  of critical varieties]\label{tie-breaking}The technical details of the following remarks can be found in \cite{Ein1997} or \cite{Lee1999}.
\begin{enumerate}\item In general, an effective divisor $G$  satisfying (1) of Definition \ref{critical} may not satisfy the condition (2) of Definition \ref{critical}. However, if $G$ is ample, then by perturbing $G$ a little bit, one can construct a new divisor $G'$ linearly equivalent to $(1+\varepsilon)G$ with $0\leq |\varepsilon| \ll1$ such that $G'$ is critical at $x$. This technique is called the tie-breaking trick.

\item Assume that $G$ is critical. Then $(X, G)$ is log canonical at the point $x$ with the minimal log canonical center $Z(G)$. Conversely, assume that $(X, G)$ is log canonical at $x$ with the minimal log canonical center $Z=Z(G)$. If $G$ is not critical at $x$, then by the tie-breaking trick, one can create a new divisor $G'$ linearly equivalent to $(1+\varepsilon)G$ with $0\leq |\varepsilon| \ll1$ such that $G'$ is critical at $x$ with the critical variety $Z(G')=Z$.

\item Let $Z$ be a critical variety at $x$. Then $Z$ is normal at $x$. In particular, if $Z$ is a curve, then $Z$ is smooth at $x$.
\end{enumerate}
\end{rmk}

An invariant measuring the difficulty of creating a new divisor $G'$ with $Z(G')\subsetneq Z(G)$ is the deficit $\df_x(G)$ of $G$ at $x$. This concept was introduced by Ein and Helmke independently (see \cite{Ein1997} and \cite{Helmke1997}). We follow Ein's definition (Section 4, \cite{Ein1997}) with an adaptation of Helmke's definition of ``wildness" (Section 3, \cite{Helmke1999} ).

Let $\pi : Y'\to X$ be the blowing-up of $x$ and $E$ be the exceptional divisor. Let $g : Y\to Y'$ be a log resolution of $\pi^*(G)+E$. We then have a log resolution $f=g\circ\pi : Y\to X$. Write $K_Y-f^*(K_X+G)=\sum a_jF_j$ and $g^*(E)=\sum e_jF_j$.

\begin{Def}\label{def:deficit}
If $x\not\in Z(G)$, we define the {\it deficit} of $G$ at $x$ as $$\df_x(G):=\inf_{f(F_j)=x}\{\dfrac{a_j+1}{e_j}\},$$ where $f$ varies among all log resolutions factoring through $\pi$. If $G$ is log canonical at $x$, we define
\[\df_x(G):=\lim_{t\to 0^+}\df_x((1-t)G).\]

Let $Z$ be a subvariety of $X$ containing $x$ and $G$ be an effective divisor such that $(X, G)$ is log canonical at $x$. We define the {\it relative deficit} of $G$ over $Z$ at $x$ as \[\df_x^Z(G):=\sup\limits_{D}\left\{\df_x(G+D) \biggr|\parbox{0.525\textwidth}{$(X,G+D)$ is log canonical at $x$, $Z(G+D)=Z$, and $D$ is an effective $\mb{Q}$-divisor}\right\}.\] If there is no effective divisor $D$ such that $(X,G+D)$ is log canonical at $x$ with $Z(G+D)=Z$, we define $\df_x^Z(G)=0$. In the case that $G=0$, we write $\df_x^Z$ for $\df_x^Z(0)$.
\end{Def}

\begin{rmk} Assume that $G$ is log canonical at $x$.  We note that the deficit is the same as Helmke's definition of local discrepancy  \cite[Definition in Section 2]{Helmke1999}:
  \[\begin{split}\df_x(G)
  &=\sup\limits_{D}\left\{\ord_xD\biggr|\parbox{0.5\textwidth}{$(X, G+D)$ is log canonical at $x$, where $D$ is an effective $\mb{Q}$--divisor}\right\}.\end{split}\]
  (See \cite[Definition in Section 3]{Helmke1997} and \cite[Section 4]{Ein1997} for equivalent characterizations.)
  From this interpretation, we see that the definition of deficit is independent of the choice of log resolution.  
 If $Z=Z(G)$, then $\df_x^Z(G)=\df_x(G)$.
\end{rmk}

The following lemma, based on Siu's idea, shows when we can construct a new divisor $G'$ with $Z(G')\subsetneq Z(G)$.

\begin{lemma}[{\cite[Lemma 4.6]{Ein1997}}]\label{lemma:Ein-Siu}Let $X$ be a smooth projective variety with an ample line bundle $L$. 
\begin{enumerate}\item Let $G$ be an effective $\mathbb{Q}$--Cartier divisor on $X$ with critical variety $Z$ at $x\in X$.
\item Let $B$ be an effective $\mathbb{Q}$--Cartier divisor on $Z$ which is linearly equivalent to $qL|_Z$ for some positive rational number $q$.
\item Assume that $\ord_xB>\df_x(G)$. 
\end{enumerate}
Then there is an effective $\mathbb{Q}$--divisor $D$ linearly equivalent to $qL$ on $X$ such that $D|_Z=B$ and a new divisor $G'=(1-t)G+ D'$ for some positive number $t<1$ such that $G'$ is critical at $x$ with $Z((1-t)G+ D')\subsetneq Z$, where $D'$ is a small perturbation of $sD$ for {some} $0<s\leq 1$. More precisely, $D'=s'D+ H$ with $H$ a $\mathbb{Q}$--divisor linearly equivalent to $\delta L$, where $\delta$ is a sufficiently small positive number such that $t$ and $s-s'$ are  also sufficiently small positive numbers.

Moreover,  \[\df_x(G')\leq \df_x((1-t)G)-\ord_xD'|_Z.\]
\end{lemma}

In practice, we have the following induction criterion due to Helmke which can be viewed as a consequence of the above lemma.

\begin{prop}[{\cite[Proposition 3.2]{Helmke1997}}]\label{Prop:Helmke'sInductionCriterion}
       Let $X$ be a smooth projective variety,  $L$ be an ample line bundle over $X$ and $G$ be a $\mathbb{Q}$-divisor linearly equivalent to $ \lambda L$ for some positive rational number $\lambda<1$. Assume that $G$ is critical at $x$ with $\df_x(G)$. Let $Z$ be the critical variety of $G$ at $x$ and $d=\dim Z>0$. If
    \begin{equation}\label{ineq}L^d\cdot Z>(\dfrac{\df_x(G)}{1-\lambda})^d\cdot \mult_x Z,\end{equation}
then there is a $\mathbb{Q}$-divisor $G'$ linearly equivalent to $\lambda'L$ with $\lambda<\lambda'<1$ such that $G'$ is critical at $x$ with the critical variety $Z'$ which is properly contained in $Z$ and  $$\dfrac{\df_x(G')}{1-\lambda'}<\dfrac{\df_x(G)}{1-\lambda}.$$
      \end{prop}

From the Helmke's criterion, we see the importance of controlling upper bounds for the deficit $\df_x(G)$. Thanks to Ein and Helmke, we have some very useful upper bounds for $\df_x(G)$.

\begin{prop}[\cite{Ein1997}, \cite{Helmke1999}]\label{prop:Ein-Helmke}~
  \begin{enumerate}
    \item\label{prop:Ein-Helmke-1}  $\df_x(G)\leq n-\ord_x G.$ {(See \cite[Proposition 4.1]{Ein1997}.)}
    \item If $G$ is critical at $x$, then the critical variety is of dimension zero at $x$ if and only if $\df_x(G)=0$. {(See\cite[Proposition 4.1]{Ein1997}.)}
\item\label{part-3}  Assume that $G$ is critical at $x$ and $x\in Z$ where $Z$ is a subvariety of $X$. Let $D$ be an effective $\mb{Q}$-divisor such that $Z$ is not in the support of $D$ and $\df_x^Z(G+D)>0$. Then $\df_x^Z(G+D)\leq \df_x^Z(G)-\ord_xD$.  {(See \cite[Lemma 3.2]{Helmke1999}.)}
  \end{enumerate}
\end{prop}

By Proposition \ref{prop:Ein-Helmke} \eqref{prop:Ein-Helmke-1} and the definition of relative deficit, we know that if the critical variety $Z$ of $G$ is a hypersurface, then $$\df^Z_x(Z)=\df_x(Z)\leq n-\ord_xZ.$$

\begin{prop}[{\cite[Propostion 4.2]{Ein1997}}] Assume that $G$ is critical at $x$ and $Z$ is the critical variety of $G$ at $x$. If $\df_x(G)\geq 1$, then $\df_x(G|_H)=\df_x(G)-1$ and $G|_H$ is critical at $x$, where $H$ is a general hypersurface in $X$ passing through $x$. In particular,  $\df_x(G)\leq \dim Z$.
\end{prop}

We also note that an implicit upper bound is given by the inequality in the following theorem.
\begin{theorem}[{\cite[Theorem 4.3]{Helmke1997}}]\label{Thm:boundofmult}
 Assume that $G$ is critical at $x$ and $Z$ is the critical variety of $G$ at $x$. Let $e$ be the embedding dimension of $Z$ at $x$, $d=\dim Z$ and $m=\mult_xZ$. Then
  \[m\leq \binom{e -\ceil{\df_x(G)}}{e-d}.\]
  In particular, if $d\geq 1$, then \[m\leq \binom{e -1}{d-1}.\] \end{theorem}

\begin{rmk}Helmke (Example 3.5 \cite{Helmke1999}) shows that when the critical variety $Z$ is a hypersurface, then $\df_x(Z)=\dim X-\mult_xZ$ which shows that the inequality in the above theorem is optimal.
\end{rmk}

Inspired by Theorem \ref{Thm:boundofmult}, we define an integer $\alpha_{d, e}(m)$ as follows and show that it is an upper bound of the deficit function.

\begin{Def}\label{Def:alpha}
Given positive integers $e, d, m$ satisfying $d\leq e$ and  $m\leq \binom{e}{d}$, we define $\alpha_{d,e}(m)$ to be the largest integer $y$ such that $\binom{e -y}{e-d}\geq m$. 
\end{Def}
Note that $0\leq \alpha_{d,e}(m)\leq d$. 

\begin{cor}\label{deflesalpha}
Assume that $G$ is critical at $x$ and $Z$ is the critical variety of $G$ at $x$. Let $e$ be the embedding dimension of $Z$ at $x$, $d=\dim Z$ and $m=\mult_xZ$. Then $\df_x(G)\leq \alpha_{d, e}(m)$.
\end{cor}
\begin{proof}
We note that \[\binom{e -y}{ e-d}-m\]
is a decreasing function of $y$ over the interval $[0, d]$. Therefore, by Theorem \ref{Thm:boundofmult}, we see that  \[\df_x(G)\leq \alpha_{d, e}(m).\]
\end{proof}

It is clear that $\alpha_{d,e}(m)\leq d-1$ if $m=\mult_xZ\geq 2$.  In fact, we know that the larger the number $m$ the smaller the integer $\alpha_{d,e}(m)$.

Corollary \ref{deflesalpha} together with Proposition \ref{prop:Ein-Helmke} \eqref{prop:Ein-Helmke-1} implies one of the practically useful results of the paper.

\begin{lemma}\label{main-lemma}
Let $X$ be a smooth projective variety of dimension $n$ and $x\in X$ a point. Assume that $L$ is an ample line bundle on $X$. Let $G$ be an effective $\mathbb{Q}$-divisor linearly equivalent to $\lambda L$ for $\lambda<1$ and critical at $x$. 

If $\ord_xG=\lambda\sigma$ for some $\sigma>n$, then we have
\[ \dfrac{\df_x(G)}{1-\lambda}\leq \dfrac{\sigma \alpha_{d, e}(m)}{\sigma-n+\alpha_{d, e}(m)},\]
where $m=\mult_xZ(G)$.

Moreover, if $L^n>\sigma_n^n\geq n^n$ for some positive number $\sigma_n$, then there exists an effective $\mathbb{Q}$-divisor $G''$ linearly equivalent to $\lambda L$ for some number $\lambda<1$ such that $G''$ is critical at $x$,  $\ord_x G''=\lambda\sigma$ with $\sigma>\sigma_n$, and
\[ \dfrac{\df_x(G'')}{1-\lambda}< \dfrac{\sigma_n \alpha_{d, e}(m)}{\sigma_n-n+\alpha_{d, e}(m)}.\]
\end{lemma}

\begin{proof}
By Proposition \ref{prop:Ein-Helmke} and Corollary \ref{deflesalpha}, we have the following inequality
\[\df_x(G)\leq\min\{\alpha_{d, e}(m), n-\ord_xG\}.\]
We argue by comparing $\alpha_{d, e}(m)$ with $n-\ord_xG$.
\begin{enumerate}[label= Case \arabic*.]
\item 
If  $\alpha_{d, e}(m)\leq n-\ord_xG= n-\lambda \sigma$, then
\[1-\lambda\geq \dfrac{\sigma-n+\alpha_{d, e}(m)}{\sigma}.\] Therefore,
\[\dfrac{\df_x(G)}{1-\lambda}\leq \dfrac{\alpha_{d, e}(m)}{1-\lambda}\leq \dfrac{\sigma\alpha_{d, e}(m)}{\sigma-n+\alpha_{d, e}(m)}.\]
\item If  $\alpha_{d, e}(m)\geq n-\ord_xG= n-\lambda \sigma$,  then $\lambda\geq \frac{n-\alpha_{d, e}(m)}{\sigma}$. Therefore,
\[\scriptstyle \dfrac{\df_x(G)}{1-\lambda}\leq \dfrac{n-\lambda\sigma}{1-\lambda} \leq \dfrac{~n-\sigma\dfrac{n-\alpha_{d, e}(m)}{\sigma}~}{~1-\dfrac{n-\alpha_{d, e}(m)}{\sigma}~}=\dfrac{\sigma\alpha_{d, e}(m)}{\sigma-n+\alpha_{d, e}(m)}.\]
\end{enumerate}

Now we prove the second inequality in Lemma \ref{main-lemma}. By the assumption $L^n>\sigma_n^n\geq n^n$, we may find a positive number $\sigma'$ such that $\sqrt[n]{L^n}>\sigma'>\sigma_n$. Since $L^n>(\sigma')^n$, {by \cite[Proposition 1.1.31]{Lazarsfeld2004},} there is an effective $\mathbb{Q}$-divisor $D$ linearly equivalent to $L$ such that $\ord_xD>\sigma'$. Let $c$ be the log canonical threshold $\lct_x(X,D)$ at $x$ and $G'=cD$. By applying the tie-breaking technique to $G'$, we may find a $\mathbb{Q}$-divisor $G''$ linearly equivalent to $\lambda L$ with $|\lambda-c|\ll 1$ such that $G''$ is critical at $x$ and $\ord_x G''=\lambda\sigma$ with $\sigma>\sigma_n$. Since $\frac{\sigma\alpha_{d, e}(m)}{\sigma-n+\alpha_{d, e}(m)}$ is a decreasing function of $\sigma$, we have the strict inequality
\[ \dfrac{\df_x(G'')}{1-\lambda}\leq\dfrac{\sigma \alpha_{d, e}(m)}{\sigma-n+\alpha_{d, e}(m)}< \dfrac{\sigma_n \alpha_{d, e}(m)}{\sigma_n-n+\alpha_{d, e}(m)}.\]
\end{proof}

The same idea used in the above proof together with Lemma \ref{lemma:Ein-Siu} leads to the following result {which will be used in case (4) in the proof of Theorem \ref{Thm-dim-5}.} 

\begin{prop} \label{smoothdivisorreduction}
Let $G''$ be the effective divisor and $\sigma_n$ be the positive number in Lemma \ref{main-lemma}.  Assume in addition that the critical variety $Z(G'')$ is smooth at $x$ and $L^d\cdot Z( G'')\geq \sigma_n^d$,  where $d=\dim Z(G'')$. Then there is an effective divisor $G_1$ linearly equivalent to $\lambda_1L$ with $\lambda_1<1$ such that $G_1$ is critical at $x$ and $d'=\dim Z(G_1)<\dim Z(G)$. Moreover, 
\[\frac{\df_x(G_1)}{1-\lambda_1}< \frac{\sigma_n\alpha'}{\sigma_n-n+\alpha'},\] where $m'=\mult_xZ(G_1)$ and $\alpha'=\alpha_{d', e'}(m')$.
\end{prop}
\begin{proof} By Lemma \ref{main-lemma}, we know that 
\[\dfrac{\df_x(G'')}{1-\lambda}< \dfrac{\sigma_n \alpha_{d, e}(m)}{\sigma_n-n+\alpha_{d, e}(m)}\leq \sigma_n,\]

Hence $\sigma_n(1-\lambda)>\df_x(G'')$. Let $\varepsilon_1$ be a sufficiently small positive number such that $\varepsilon_1<\min\{1, \frac{\lambda(\sigma-\sigma_n)}{{8}}, \sigma_n-\frac{\df(G'')}{1-\lambda}\}$. Since $L^d\cdot Z( G'')\geq \sigma_n^d$, there exists an effective  $\mathbb{Q}$-divisor $B$ on $Z$ linearly equivalent to $(1-\lambda)L|_Z$ such that \[\ord_xB=(1-\lambda)(\sigma_n-\varepsilon_1)>\df_x(G'').\]

As a result, we can apply Lemma \ref{lemma:Ein-Siu} to $B$ and $G''$ and get a new divisor $G_1=(1-t){G''}+D'$ with $\ord_xD'|_Z\geq s'(1-\lambda)(\sigma_n-\varepsilon_1)+\delta_0$ for some real numbers $0<t<1$, $0<s'<1$ and nonnegative $\delta_0$. In particular, $G_1$ is critical at $x$ with $\dim Z(G_1)<\dim Z( G'')$. Moreover, $G_1$ is linearly equivalent to $\lambda_1 L$ with $\lambda_1=(1-t)\lambda+ s'(1-\lambda)+\delta<1$. The value of $t$ can be arbitrarily small as long as $\delta$ is sufficiently small.  We can then choose {$\delta< \frac{\lambda(\sigma-\sigma_n)}{4(\sigma_n+1)}$} such that $t<\min\{\frac{1}{2}, \lambda\}$ and $\df_x((1-t)G'')-\df_x(G'')<\frac{\varepsilon_1}{2}$. Let $\varepsilon_2$ be a sufficiently small positive number such that $\df_x((1-t)G'')-\df_x( G'')<\varepsilon_2<\frac{\varepsilon_1}{2}$.
Recall that  $$\df_x(G_1)\leq \min\{\alpha', \df_x( G'')+\varepsilon_2-\ord_xD'|_Z\},$$ where  $m'=\mult_xZ(G_1)$ and  $\alpha'=\alpha_{d', e'}(m')$. Notice that 
\[\scriptstyle\begin{split}
~&~\df_x(G'')+\varepsilon_2-\ord_xD'|_Z\\
\leq ~&~ \df_x((1-t)G'')+\varepsilon_2-\ord_xD'|_Z\\ 
\leq ~&~ n-(1-t)\ord_xG''-\ord_xD'|_Z+\varepsilon_2\\ 
\leq ~&~ n- (1-t)\lambda\sigma-\big(s'(1-\lambda)(\sigma_n-\varepsilon_1)+\delta_0\big)+\varepsilon_2\\
= ~&~ n-(1-t)\lambda\sigma_n-(1-t)\lambda(\sigma-\sigma_n)-\big(s'(1-\lambda)(\sigma_n-\varepsilon_1)+\delta_0\big)+\varepsilon_2\\
< ~&~ n- (1-t)(\lambda\sigma_n+4\varepsilon_1)-\frac{1}{2}(1-t)\lambda(\sigma-\sigma_n)-\big(s'(1-\lambda)(\sigma_n-\varepsilon_1)+\delta_0\big)+\varepsilon_2\\
~&\hcell{\hfill \text{~(because~$ \varepsilon_1<\tfrac{\lambda(\sigma-\sigma_n)}{8}$)} }\\
<~&~ n-\lambda_1(\sigma_n+\varepsilon_1)+\delta(\sigma_n+\varepsilon_1)-\frac{1}{4}\lambda(\sigma-\sigma_n)\\
\phantom{=} ~&~  +\big((1-t)\lambda+2s'(1-\lambda)-4(1-t)\big)\varepsilon_1-\delta_0+\varepsilon_2  \\
~&\hcell{\hfill \text{~(because~ $1-t>\tfrac12$})}\\
<~&~n-\lambda_1(\sigma_n+\varepsilon_1)+\big((1-t)\lambda+2s'(1-\lambda)-4(1-t)\big)\varepsilon_1-\delta_0+\varepsilon_2  \\
~&\hcell{\hfill ~\text{(because ~ $\delta(\sigma_n+\varepsilon_1)<\delta(\sigma_n+1)<\tfrac{\lambda(\sigma-\sigma_n)}{4}$)} }\\
<~&~n-\lambda_1(\sigma_n+\varepsilon_1)-(1-t)\varepsilon_1-\delta_0+\varepsilon_2\\
~&\hcell{\hfill  ~(\text{because ~} t<\lambda<1 \text{~and~} s'<1)}\\
= ~&~ n- \lambda_1(\sigma_n+\varepsilon_1)-\big( (1-t)\varepsilon_1-\varepsilon_2\big)-\delta_0\\
< ~&~ n- \lambda_1(\sigma_n+\varepsilon_1)-\big(\frac12\varepsilon_1-\varepsilon_2\big)\\
< ~&~ n- \lambda_1(\sigma_n+\varepsilon_1).
\end{split}
\] 
Hence, $\df_x(G_1)\leq \min\{\alpha', n-\lambda_1(\sigma_n+\varepsilon_1)\}$. Similar to the proof of Lemma \ref{main-lemma}, we argue by comparing $\alpha'$ with $n-\lambda_1(\sigma_n+\varepsilon_1)$. 
\begin{enumerate}[label= Case \arabic*.]
\item 
If  $\alpha'\leq n-\lambda_1(\sigma_n+\varepsilon_1) $, then $\lambda_1\leq \frac{n-\alpha'}{\sigma_n+\varepsilon_1}$. Hence \[\dfrac1{1-
\lambda_1}\leq \dfrac{\sigma_n+\varepsilon_1}{\sigma_n+\varepsilon_1-n+\alpha'}\] and
\[\dfrac{\df_x(G_1)}{1-\lambda_1}\leq \dfrac{\alpha'}{1-\lambda_1}\leq \dfrac{\alpha'(\sigma_n+\varepsilon_1)}{\sigma_n+\varepsilon_1-n+\alpha'}<\dfrac{\alpha'\sigma_n}{\sigma_n-n+\alpha'}.\]
The last inequality follows from the fact that the function $f(x)=\frac{\alpha' x}{x-n+\alpha'}$ is decreasing.
\item If  $\alpha'\geq n-\lambda_1(\sigma_n+\varepsilon_1) $, then $\lambda_1\geq \tfrac{n-\alpha'}{\sigma_n+\varepsilon_1}$. Therefore, \[\begin{split}\dfrac{\df_x(G_1)}{1-\lambda_1}\leq \dfrac{n-\lambda_1(\sigma_n+\varepsilon_1)}{1-\lambda_1}\leq&
\dfrac{n-\dfrac{n-\alpha'}{\sigma_n+\varepsilon_1}(\sigma_n+\varepsilon_1)}{1-\dfrac{n-\alpha'}{\sigma_n+\varepsilon_1}}\\
=&
\dfrac{\alpha'(\sigma_n+\varepsilon_1)}{\sigma_n+\varepsilon_1-n+\alpha'}<\dfrac{\alpha'\sigma_n}{\sigma_n-n+\alpha'}.\end{split}\]
The second inequality follows from the fact that the function $f(x)=\frac{n-x(\sigma_n+\varepsilon_1)}{1-x}$ is decreasing. 
The last inequality follows from the fact that the function $f(x)=\frac{\alpha' x}{x-n+\alpha'}$ is decreasing.
\end{enumerate}
\end{proof}

It will be very helpful to know the singularities of $Z$, especially the multiplicity $\mult_xZ$. In \cite{Kawamata1997}, Kawamata initiated the study of subadjucntion formulas.  For critical varieties, now we have the following characterization on their singularities.

\begin{theorem}[\cite{Fujino2012}]\label{thm:Kawamata1998}
 Assume that $G$ is critical at $x$ with the critical variety $Z$.  Then there exist an effective $\mb{Q}$-divisor $D_Z$ on $Z$ such that $$(K_X+D)|_Z=K_Z+D_Z$$ and the pair $(W, D_W)$ is klt at $x$. In particular, $Z$ has at most a rational singularity at $x$.
\end{theorem}

\begin{rmk}\label{rmk:embedding-dim-surface-sing}
Let $Z$ be a critical variety and $x\in Z$ be a point. Assume that $\dim Z=2$. Since $x$ is a rational singularity,  then $\mult_xZ=e-1$, where $e$ is the embedding dimension of $Z$ at $x$.
\end{rmk}

\section{Volumes of graded linear series and applications \label{Sec-volume}}
Let $X$ be a smooth projective variety of dimension $n$,  $\mathfrak{m}_x$ be the maximal ideal of a closed point $x\in X$ and $L$ be an ample line bundle on $X$. Let $q$ be a nonnegative real number. Then $A^{q k}_k : =\{H^0(X, kL\otimes \mathfrak{m}_x^{\ceil{q k}})\}$ is a graded sublinear series. We define the {\it volume function} of this sublinear series as $\vol(q,L):=\lim\limits_k\{\frac{n!}{k^n}\dim A^{q k}_k\}$. 
For any two real numbers $\beta\leq \gamma$, we write $\vol(\beta,\gamma,L)$ for the difference $\vol(\beta,L)-\vol(\gamma,L)$. 
We denote by $F_k(q)$ the fixed part of $|kL\otimes \mathfrak{m}_x^{\ceil{kq}}|$ and define $\phi_k(q)=q-\frac{\ord_x F_k(q)}{k}$ if $|kL\otimes \mathfrak{m}_x^{\ceil{kq}}|\neq\emptyset$ and $\phi_k(q)=-\infty$ otherwise. The function $\phi(q)=\sup\limits_k\{\phi_k(q)\}$ is called the {\it mobility in codimension one} of $L$ at $x$. We note that $\phi(q)=\lim\limits_k\{\phi_k(q)\}$, since $X$ is smooth. One very useful property of $\phi(q)$ is the following.

\begin{prop}[{\cite[Lemma 4.1]{Helmke1999}}] The function
  $\phi(q)$ is a concave down function on $\mb{R}_{\geq0}$.
\end{prop}

The following result, the idea of its proof is due to Fujita (see \cite{Fujita1993}), has appeared in \cite{Kawamata1997}, \cite{Helmke1999} and \cite{Lee1999} in different forms.
	
\begin{prop}\label{integral}
  Let $\beta$ and $\gamma$ be two positive numbers such that $\gamma>\beta$. Assume that ${\vol(t,L)}\geq0$ for any $t$ such that $\gamma\geq t\geq\beta$.
    $$\vol(\beta, \gamma, L)\leq n\int_\beta^\gamma\phi(t)^{n-1}\d t.$$
\end{prop}
\begin{proof}
Let $k$ be a sufficiently divisible integer and $t$ be a rational number. We may assume that $kt$ is an integer, denoted by $j$. 
Let $\pi: Y\to X$ be the blow up of $X$ at $x$, $E$ be the exceptional divisor and  $\widetilde{F_{k}(j)}$ be the strict transform of the fixed part $F_{k}(j)$.  Then $\pi^*(F_k(j))=\widetilde{F_{k}(j)}+\ord_xF_k(j)E$ and we have the following commutative diagram
\begin{center}
\begin{tikzpicture}[baseline= (a).base]
\node[scale=0.75] (a) at (0,0){
	\begin{tikzcd}[column sep=small]
		& & {H^0(Y, \O_Y(\pi^*(kL)-\widetilde{F_{k}(j)}-jE))}\arrow{r}{\psi}\arrow{d}{\varphi}&  {H^0(E, \O_E(j-\ord_xF_k(j)))}\arrow{d}{\theta}\\
		0\arrow{r} & {H^0(Y, \O_Y(\pi^*(kL)-(j+1)E))}\arrow{r} & { H^0(Y, \O_Y(\pi^*(kL)-jE))}\arrow{r}{\rho}& { H^0(E, \O_E(j))}.
	\end{tikzcd}
};
\end{tikzpicture}
\end{center}
where $\psi$ and $\rho$ are restrictions, and $\varphi$ and $\theta$ are injections. Because $F_k(j)$ is the fixed part, we know that $\varphi$ is an isomorphism.

Then
\[\begin{split}
&h^0(X, \O_X(kL)\otimes \mathfrak{m}_x^{j})-h^0(X, \O_X(kL)\otimes \mathfrak{m}_x^{j+1})\\
=& h^0(Y, \O_Y(\pi^*(kL)-jE))-h^0(Y, \O_Y(\pi^*(kL)-(j+1)E))\\
= &\dim \im(\rho)=\dim \im(\rho\circ\varphi) \leq  \dim \im(\theta)
= h^0(E, \mc{O}_{E}(j-\ord_xF_k(j)))\\
\leq & \dfrac{(j-\ord_xF_k(j))^{n-1}}{(n-1)!}.
\end{split}
\]
Therefore
$$\vol(\beta,\gamma, L) \leq  \displaystyle\lim_{k\to\infty}\dfrac{\frac{\displaystyle\sum_{j=\ceil{k\beta}+1}^{\ceil{k\gamma}}(j-\ord_xF_k(j))^{n-1}}{(n-1)!}}{\frac{k^n}{n!}}=n\int_\beta^\gamma \phi(t)^{n-1}\d t.$$
\end{proof}

The idea behind the proofs of the following two propositions comes essentially from Helmke \cite{Helmke1999}.

\begin{prop}\label{optimizingboundfordivisor}
Let $L$ be an ample line bundle on $X$ with $L^n>\sigma^n\geq n^n$ and $x\in X$. Assume that for some rational number $q>\sigma$ the linear system $|kL\otimes \mathfrak{m}_x^{kq}|$ is nonempty for a sufficiently large $k$. 
There exists an effective $\mathbb{Q}$-divisor $G$ linearly equivalent to $\lambda L$ for some positive $\lambda<1$ such that it is critical at $x$ with $\ord_x G>\lambda\sigma$.
Furthermore, if the critical variety $Z=Z(G)$ is a divisor in $X$ with $\mult_xZ=m$. Then 
\[\frac{\df_x (G)}{1-\lambda} <\frac{n-m}{1-\lambda}-\frac{\lambda}{1-\lambda}\phi_k(q).\] 
In particular, if $\phi(q)>(n-m)-\mu(q-\sigma)$ for some positive number $\mu$, then 
\[\frac{\df_x(G)}{1-\lambda}<\frac{n-m+\mu\sigma}{1+\mu}.\]
\end{prop}
\begin{proof}
Let $D$ be a general element in $|kL\otimes \mathfrak{m}_x^{kq}|$. Let $c=\lct(X,\frac Dk)<1$ and $G'=\frac{c}{k}D$. Hence $G'$ is linearly equivalent to $cL$ with $\ord_xG'\geq cq$. If $G'$ is critical at $x$, then we let $G=G'$ and $\lambda=c$. Otherwise, we can perturb $G'$ to get the desired divisor $G$ as follows. By tie breaking, for any $0\leq\alpha\ll 1$ and some $0\leq\epsilon\ll 1$, there is an effective $\mathbb{Q}$-divisor $H$ linearly equivalent to $\alpha G$ such that  $G=(1-\epsilon)(G'+H)$ is critical at $x$. Hence $G$ is linearly equivalent to $\lambda L$, where $\lambda=(1-\epsilon)(1+\alpha)c$. When $\alpha$ is sufficiently small, we have $\lambda<1$ and $\ord_xG\geq(1-\epsilon)\ord_xG'\geq (1-\epsilon)cq>\lambda\sigma$.

Now assume that the critical variety $Z$ is a divisor. We write $$G=(1-\epsilon)(G'-\frac{c}{k}F_k(q))+((1-\epsilon)(\frac{c}{k}F_k(q)+H)).$$
Since the critical variety $Z$ is a divisor, we know that $G-Z$ is effective. By the choice (generality) of $G'$, the divisor $(1-\varepsilon)(\frac{c}{k} F_k(q)+H)-Z$ is also effective. Hence \[\df_x^Z((1-\varepsilon)(\frac{c}{k}F_k(q)+H))\leq \df_x^Z(Z).\]
By Proposition \ref{prop:Ein-Helmke} (3),  we have
\begin{equation}\label{eq:def-optimized}\begin{split}
\df_x(G)&{=\df_x^Z(G)}\\
&\leq \df_x^Z( (1-\epsilon)(\frac{c}{k}F_k(q)+ H))-\ord_x((1-\epsilon)(G'-\frac{c}{k}F_k(q)))\\
&\leq \df_x^Z(Z)-(1-\epsilon)\ord_x(G'-\frac{c}{k}F_k(q))\\
&\leq \df_x^Z{(Z)}-(1-\epsilon)c\phi_k(q)\\
&\leq (n-m)-(1-\epsilon)c\phi_k(q).\end{split}\end{equation}
Recall that $\lambda=(1-\epsilon)(1+\alpha)c$, $c<1$ and $0\leq\alpha\ll 1$. Then $(1+\alpha)c<1$ and we obtain that
\begin{equation}\label{ineq-3}
\begin{split}\frac{\df_x(G)}{1-\lambda}&{\leq  \frac{(n-m)-(1-\epsilon)c\phi_k(q)}{1-\lambda} \quad\quad{\text{(by the inequality~ \eqref{eq:def-optimized})}}}\\
&{=} \frac{(n-m)}{1-\lambda}-\frac{1}{(1+\alpha)c}\frac{\lambda}{1-\lambda}\phi_k(q)\\
&<\frac{(n-m)}{1-\lambda}-\frac{\lambda}{1-\lambda}\phi_k(q).\end{split}
\end{equation}
The first inequality in Proposition \ref{optimizingboundfordivisor} is proved. 

By the hypothesis on $\phi(q)$, there exist $k$ sufficiently large such that
$\phi_k(q)>(n-m)-\mu(q-\sigma)$. We obtain
\begin{equation}\label{eq-optimizingboundfordivisor}
\df_x(G)<(1-\lambda)(n-m)+\lambda\mu(q-\sigma)
\end{equation}
{by the inequality \eqref{ineq-3}.}

Note that by Proposition \ref{prop:Ein-Helmke} (1) we have $\df_x(G)\leq n-\lambda q\leq \sigma-\lambda q$. Then \[\frac{\lambda}{1-\lambda}(q-\sigma)\leq \sigma-\frac{\df_x(G)}{1-\lambda}.\] Therefore, from the inequality \eqref{eq-optimizingboundfordivisor}, we get
\[p=\frac{\df_x(G)}{1-\lambda}< (n-m)+\mu(\sigma-p).\] Solving this inequality for $p$, we see that
\[p< \frac{n-m+\mu\sigma}{1+\mu},\]
\end{proof}

\begin{prop}\label{Best-mu}Assume that $L^n>\sigma^n\geq n^n$. For a real number $w$  in $[0,\frac{\sigma(n-1)}{n})$, we set $\mu(w)$ to be the minimal positive number satisfying
 \begin{equation}\label{min-slope}
 \frac{(\frac{w}{\sigma}+\mu)^n}{\mu(1+\mu)^{n-1}}\leq 1.
 \end{equation}
Then there exists a rational number $q>\sigma$ such that
\begin{equation}\label{condonq}
\phi(q)> w-\mu(w)(q-\sigma)
\end{equation}
for all numbers $w\in [0,\frac{\sigma(n-1)}{n})$.
 In particular, there exists a rational number $q>\sigma$ such that  \[\phi(q)> w- \dfrac{w}{\sigma(\sigma-1-w)}(q-\sigma)\]
 for all $w\in [0,n-1)$.
 \end{prop}

\begin{proof}
Assume contrarily that for every rational number $q>\sigma$, there exists some $w$ such that the inequality \eqref{condonq} fails. For simplicity, we write $l_w(q)=w-\mu(w)(q-\sigma)$. We claim that there is a real number $w\in[0,\frac{\sigma(n-1)}{n})$ such that $\phi(q)\leq l_{w}(q)$ for any rational number $q>0$.

We define 
\[\psi(q):=\sup\limits_{w\in[0,\frac{\sigma(n-1)}{n})}\{l_w(q)\}
\]for $q>0$. As the supreme of a family of linear functions, the function $\psi(q)$ is a concave up function of $q$ with $\psi(\sigma)=\frac{\sigma(n-1)}{n}$. We also note that $\mu(w)\to+\infty$ as $w\to \frac{\sigma(n-1)}{n}$. That implies $\psi(q)=\infty$ for $q<\sigma$. Since we assume contrarily that $\phi(q)\leq w-\mu(w)(q-\sigma)$ for every rational number $q>\sigma$, we see that $\psi(q)\geq \phi(q)$ for all $q>0$. Recall that $\phi(q)$ is concave down and $\psi(q)$ is concave up for $q\geq \sigma$. Then $\phi(q)$ and $\psi(q)$ are separated by a tangent line of $\psi(q)$ (see Figure \ref{Helmke-fig}). 
This fact proves our claim that there is a $w\in [0,\frac{\sigma(n-1)}{n})$ such that $\phi(q)\leq l_w(q)=:l(q)$ for every $q>0$.  
 
\begin{figure}[htbp]
\begin{center}
\pgfplotsset{every plot/.append style={line width=4pt}}
\begin{tikzpicture}[/pgf/declare function={
    psi=2.025+x/4-2*sqrt(3/4*(x-4));
    l=3/2-1/4*(x-4);
    diag=x;
    },
    ]
]
\begin{axis}[
  axis lines=middle,
  unit vector ratio*=1 1 1,
  xmin=-0.5,
  xmax=14,
  ymin=-0.5,
  ymax=6,
  xtick={2,4,10},
  xticklabels={$a$, $\sigma$,$b$},
  xlabel={$q$},
  ytick={2,3, 4},
  yticklabels={$a$, $\sigma-\sigma/n$, $\sigma$},
  yticklabel style = {font=\small,xshift=0.5ex},
  xticklabel style = {font=\small,yshift=0.5ex},
  xlabel style = {font=\tiny,anchor=north east}
]

\addplot[domain=1:12, samples=1000, smooth] {psi};
\node[font=\tiny] at (axis cs:7.5,1.5) {$\psi(q)$};
\addplot[domain=1:11] {l}; 
\node[font=\tiny] at (axis cs:1,2.6) {$l(q)$};
\addplot[domain=0:5, dotted] {diag}; 
\addplot[smooth] coordinates {(0,0) (0.05, 0.05) (1.25,0.95) (2.05,1.3) (2.8,1.4) (4.05,1.3) (6.1,0.925) (7.6,0.545)  (8.35,0.3) (9.2,0)};
\node[font=\tiny] at (axis cs:2.75,1) {$\phi(q)$};
\addplot[mark=none, dotted] coordinates {(4, 0) (4, 4)};
\addplot[mark=none, dotted] coordinates {(2, 0) (2, 2)};
\addplot[domain=0:4, dotted] {4}; 
\addplot[domain=0:4, dotted] {3};
\addplot[domain=0:2, dotted] {2};

\end{axis}
\end{tikzpicture}
\end{center}
\caption{Relations between $l(q)$, $\phi(q)$ and $\psi(q)$ \label{Helmke-fig}}
\end{figure}
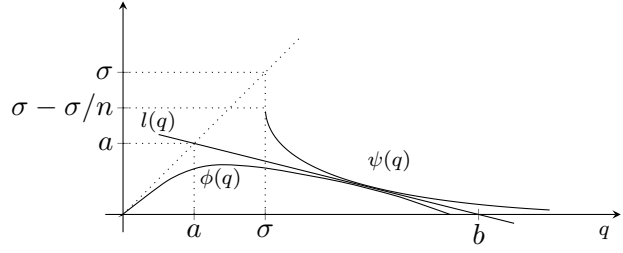

Set $\nu=\mu(w)$, $a= \frac{w+\nu\sigma}{1+\nu}\leq \sigma$ and  $b=\sigma+\frac{w}{\nu}\geq \sigma$. \[l(a)=w- \nu(\frac{w+\nu\sigma}{1+\nu}-\sigma)=\frac{w+\nu\sigma}{1+\nu}=a,\]
\[ l(b)=w-\nu(\sigma+\dfrac w{\mu}-\sigma)=0.\]
Therefore,
\[\phi(t)\leq \min\{t,l(t)\}\leq \begin{cases} t& 0\leq t\leq a\\
l(t) & a\leq t\leq b\\
0 & t\geq b.
\end{cases}\]
By Proposition \ref{integral}, we have
\begin{align*}\label{eq:vol}
    \vol(0,b,L)&\leq \int^a_0nt^{n-1}\d t
+\int^b_an\phi(t)^{n-1}\d t\\
&\leq \int^a_0nt^{n-1}\d t
+\int^b_anl(t)^{n-1}\d t\\
&=a^n-\frac{1}{\nu}[l(b)^n-l(a)^n]\\
&=a^n(1+\frac{1}{\nu})=\sigma^n\frac{(\frac{w}{\sigma}+\nu)^n}{\nu(1+\nu)^{n-1}}.\\
&\leq \sigma^n.
\end{align*}Consequently, $\vol(b,L)\geq \vol(0,L)-\sigma^n=L^n-\sigma^n>0$.  It follows that $|kL\otimes \mathfrak{m}_x^{rk}|\neq\emptyset$ for a rational number $r>b$ and $k\gg 0$ which implies that  $\phi(r)\geq 0$.   However, we have $\phi(r)\leq l(r)< 0$. This is a contradiction.

For the last assertion, it is enough to check that for any $w\in [0,n-1)$,  $\frac{w}{\sigma(\sigma-1-w)}$ satisfying inequality   \eqref{min-slope}. Hence $\frac{w}{\sigma(\sigma-1-w)}\geq \mu(w)$ and we obtain our assertion. 
\end{proof}

\begin{cor}\label{upperbound-p}
Under the assumptions in Proposition \ref{optimizingboundfordivisor}, we have
\[\frac{\df_x(G)}{1-\lambda}<\frac{\sigma(n-m)}{\sigma-1}.\]
\end{cor}

\section{Global Generation of Adjoint Line Bundles \label{Sec-proof}}
By Proposition \ref{Prop:Helmke'sInductionCriterion}, in order to get smaller lower bounds on $L^d\cdot Z$, we want smaller upper bounds for $(\frac{\df_x(G)}{1-\lambda})^{n-1}\cdot \mult_x Z$. By Lemma \ref{main-lemma}, we see that the integer $\alpha_{d, e}(m)$ is a key factor.  In lower dimensional cases, one can easily find $\alpha_{d, e}(m)$ by solving the equation in Definition \ref{Def:alpha}. 

\begin{ex}\label{lemma:dim=2}
Assume that $Z$ is the critical variety of an effective divisor $G$ at $x$ and  $\dim Z=2$. By Remark \ref{rmk:embedding-dim-surface-sing}, we know that $m=\mult_xZ=e-1$, where $e$ is the embedding dimension. Then by Theorem \ref{Thm:boundofmult}, we have the following inequality
$$m\leq \begin{pmatrix} m+1 -\alpha_{2, e}(m)\\ m-1\end{pmatrix}.$$
If $m\geq 2$, then $def_p(G)\leq \alpha_{2, e}(m)= 1$.
\end{ex}

\begin{ex}\label{lemma:dim=3}Let $G$ be an effective divisor on a smooth projective variety $X$ of dimension $5$.
Assume that $G$ is critical at $x$ and the critical variety $Z(G)$ is of dimension $3$. If $Z(G)$ is not smooth at $x$, then
\[\df_x(G)\leq\alpha_{3, 5}(m)=\begin{cases}2, ~~\mult_xZ(G)=2, 3;\\ 1, ~~\mult_xZ(G)=4, 5, 6.\end{cases}\]
\end{ex}

By calculating the integer $\alpha_{d,e}(m)$ and applying Lemma \ref{main-lemma}, 
Corollary \ref{upperbound-p} and Proposition \ref{Prop:Helmke'sInductionCriterion}, we  prove the following effective result on global generation of adjoint line bundles on 5-folds. In the proof, we will fix $d$ in each case, therefore we will use $\alpha(m)$ for $\alpha_{d,e}(m)$. We can check that $\alpha(m)\leq \alpha_{d,n}(m)$.

\begin{theorem}\label{Thm-dim-5}
Let $X$ be a smooth projective variety of dimension 5,  $L$ be an ample line bundle on $X$ and $x\in X$ be a point. Assume that $\sqrt[5]{L^5}>7$ and $\sqrt[d]{L^d\cdot Z}\geq 7$ for any subscheme $Z\subset X$ containing $x$ of dimension  $\dim Z=d$, where $d=1, 2, 3, 4$. Then $\O_X(K_X+L)$ is globally generated at $x$.
\end{theorem}

\begin{proof}Let $\sigma=7$. By Proposition \ref{Best-mu}, there exists a rational number $q>7$ such that $\phi(q)> w-\mu(w)(q-\sigma)$ for all $w\in [0, n-1)$. By Proposition \ref{optimizingboundfordivisor}, we obtain an effective $\mathbb{Q}$-divisor $G$ linearly equivalent to $\lambda L$ with $\lambda<1$ and critical at $x$ with multiplicity $\ord_xG>\lambda\sigma$. Let $Z$ be the critical variety and $m=\mult_xZ$. If $\dim Z=0$,  then the theorem follows from  Lemma \ref{cor:0-dimb-p-f}.
\begin{enumerate}
\item\label{dim=1} Assume that $\dim Z =1$. Then $Z$ is smooth and $\alpha(m)=1$. Therefore \[\frac{\df_x(G)}{1-\lambda}< \frac{\sigma}{\sigma-5+1}\leq \frac{7}{3} <7.\] Applying Proposition \ref{Prop:Helmke'sInductionCriterion} and  Lemma \ref{cor:0-dimb-p-f}, we prove the theorem.

\item\label{dim=2} Assume that $\dim Z=2$. Then $m=\mult_xZ\leq 4$. If $m\geq 2$, then $\alpha(m)=1$ and \[\frac{\df_x(G)}{1-\lambda}< \frac{\sigma}{\sigma-5+1}\leq \frac{7}{3}<\frac{7}{\sqrt{m}}.\]  If $m=1$, then $\alpha(m)\leq 2$ and \[\frac{\df_x(G)}{1-\lambda}< \frac{2\sigma}{\sigma-5+2}\leq \frac{7}{2}< 7.\]  In both cases, we can construct a new divisor $G_1$ with the following properties: $G_1$ is linearly equivalent to $\lambda_1L$ for $\lambda_1<1$;
 $G_1$ is critical at $x$ with the critical variety $Z_1$ properly contained in $Z$;
 \[\dfrac{\df_x(G_1)}{1-\lambda_1}<\begin{cases} \frac{7}{3} & m\geq 2\\ \frac{7}{2}& m=1,\end{cases}\]
by Proposition \ref{Prop:Helmke'sInductionCriterion}.
If $Z_1$ is a point, then the theorem follows directly from Lemma \ref{cor:0-dimb-p-f}. If $Z_1$ is a curve, we apply Proposition \ref{Prop:Helmke'sInductionCriterion} again and then Lemma \ref{cor:0-dimb-p-f}.

\item\label{dim=3} Assume that $\dim Z = 3$. Then $m\leq 6$. If $m\geq 2$, then $\alpha(m)\leq 2$ and \[\frac{\df_x(G)}{1-\lambda}< \frac{2\sigma}{\sigma-5+2}\leq \frac{7}{2}<\frac{7}{\sqrt[3]{m}}.\] If $m=1$, then $\alpha(m)\leq 3$ and \[\frac{\df_x(G)}{1-\lambda}< \frac{3\sigma}{\sigma-5+3}\leq \frac{21}{5}<7.\]  In both cases, we can construct a new divisor $G_1$ with the following properties: $G_1$ is linearly equivalent to $\lambda_1L$ for $\lambda_1<1$;
 $G_1$ is critical at $x$ with the critical variety $Z_1$ properly contained in $Z$; and
 \[\dfrac{\df_x(G_1)}{1-\lambda_1}<\begin{cases}\frac{7}{2} & m\geq 2\\ \frac{21}{5} & m=1\end{cases}\]
by Proposition  \ref{Prop:Helmke'sInductionCriterion}. 

Assume that $Z_1$ is a curve. By applying Proposition \ref{Prop:Helmke'sInductionCriterion} again and then Lemma \ref{cor:0-dimb-p-f}, we are done. Assume that  $Z_1$ is a point. Applying Lemma \ref{cor:0-dimb-p-f} will do the job. 

Assume $Z_1$ is a surface. Then
\[m_1=\mult_xZ_1\leq \begin{cases}4 &m\geq 2\\
2&   m=1.\end{cases}\]

Consequently, \[\dfrac{\df_x(G_1)}{1-\lambda_1}<\frac{7}{\sqrt{m}}.\]
We apply Proposition \ref{Prop:Helmke'sInductionCriterion} again to draw our conclusion.

\item Assume that $\dim Z = 4$. Then $m\leq 4$ and $\alpha(m)\leq 5-m$.
\begin{enumerate}
\item\label{smooth-case}    If $m=1$, then by Proposition \ref{smoothdivisorreduction}, we can construct a new divisor $G_1$ critical at $x$ with the critical variety $Z_1$ properly contained in $Z$ and
 \[\frac{\df_x(G_1)}{1-\lambda_1}< \frac{\sigma\alpha'}{\sigma-5+\alpha'}.\]
Since $m'=\mult_xZ_1\leq 3$, following the arguments in cases \eqref{dim=1}, \eqref{dim=2} and \eqref{dim=3}, we know that $\frac{\df_x(G_1)}{1-\lambda_1}<\frac{7}{\sqrt{m'}}$. 
Hence we can apply Proposition \ref{Prop:Helmke'sInductionCriterion} again.

\item
If $m\geq 2$, by Proposition \ref{optimizingboundfordivisor} and Corollary \ref{upperbound-p} we know that \[\frac{\df_x(G)}{1-\lambda}<\frac{3\sigma}{\sigma-1}=\frac{7}{2}< \frac{7}{\sqrt[4]{4}}.\]
By Proposition \ref{Prop:Helmke'sInductionCriterion}, we can construct a new divisor $G_1$ such that it is critical at $x$ with the critical variety $Z_1$ properly contained in $Z$ and
 \[\frac{\df_x(G_1)}{1-\lambda_1}< \frac{7}{2}\leq \min\{7, \frac{7}{\sqrt{4}}, \frac{7}{\sqrt[3]{6}}\}.\]
Then we can apply Proposition \ref{Prop:Helmke'sInductionCriterion} repeatedly till we get a $0$-dimensional $Z$.
\end{enumerate}

Summarizing our argument, we see that $\O_X(K_X+L)$ is globally generated at $x$.
\end{enumerate}
\end{proof}

\begin{rmk}We note that the same argument as in the proof of Theorem \ref{Thm-dim-5} will give a concise  proof of Kawamata's result (Theorem \ref{thm:Kawamata}) on $4$-folds.
\end{rmk}


\end{document}